\documentclass[twocolumn]{article}
\usepackage[utf8]{inputenc}
\usepackage[left=2cm, right=2cm, top=2cm]{geometry} 
\usepackage[small]{titlesec}
\usepackage{etoolbox}
\usepackage{amsthm}
\makeatletter
\patchcmd{\ttlh@hang}{\parindent\z@}{\parindent\z@\leavevmode}{}{}
\patchcmd{\ttlh@hang}{\noindent}{}{}{}
\makeatother

\newcommand{\minimize}[3]{\begin{array}{rl}
		{\underset{#2}{\textrm{minimize}}} & \begin{aligned}[t] 
			#1
		\end{aligned}  \\[15pt]
		\textrm{subject to} &
		\begin{aligned}[t] 
			#3
		\end{aligned}
	\end{array}}

	\usepackage{mathptmx} 
	\usepackage{times} 
	\usepackage{amsmath} 
	\usepackage{amssymb}  
	\usepackage{algorithm}
	\usepackage{algpseudocode}
	\usepackage{color}
	\usepackage{makecell}
	\usepackage{subfig, float}
	\usepackage{tikz,pgfplots}
	\usepackage{epstopdf}
	\usepgfplotslibrary{external} 
	
	\newlength\figureheight 
	\newlength\figurewidth  
	
	\newtheorem{mydef}{Definition}
	
	\newtheorem{mythe}{Theorem}
	\newtheorem{mycor}{Corollary}
	\newtheorem{myrem}{Remark}
	\DeclareMathOperator*{\argmin}{arg\,min}
	\DeclareMathAlphabet{\mathcalOld}{OMS}{cmsy}{m}{n}
	\DeclareMathOperator*{\minimizer}{minimize}
	\DeclareMathOperator*{\subjecttoo}{subject\:to}

\title{\LARGE \bf
Improved Path Planning by Tightly Combining Lattice-based Path Planning and Numerical Optimal Control	 
}

\date{}
\author{Kristoffer Bergman, Oskar Ljungqvist and Daniel Axehill}

\begin{document}
	
	\algblock{ParFor}{EndParFor}
	\algnewcommand\algorithmicparfor{\textbf{parfor}}
	\algnewcommand\algorithmicpardo{\textbf{do}}
	\algnewcommand\algorithmicendparfor{\textbf{end\ parfor}}
	\algrenewtext{ParFor}[1]{\algorithmicparfor\ #1\ \algorithmicpardo}
	\algrenewtext{EndParFor}{\algorithmicendparfor}

	\maketitle
	\thispagestyle{empty}
	\pagestyle{empty}

		\textbf{\textit{Abstract ---}}\textbf{This paper presents a unified optimization-based path planning approach to efficiently compute locally optimal solutions to advanced path planning problems. The approach is motivated by first showing that a lattice-based path planner can be cast and analyzed as a bilevel optimization problem. This information is then used to tightly integrate a lattice-based path planner and numerical optimal control in a novel way. The lattice-based path planner is applied to the problem in a first step using a discretized search space, where system dynamics and objective function are chosen to coincide with those used in a second numerical optimal control step. As a consequence, the lattice planner provides the numerical optimal control step with a resolution optimal solution to the problem, which is highly suitable as a warm-start to the second step. This novel tight combination of a sampling-based path planner and numerical optimal control makes, in a structured way, benefit of the former method's ability to solve combinatorial parts of the problem and the latter method's ability to obtain locally optimal solutions not constrained to a discretized search space. Compared to previously presented combinations of sampling-based path planners and optimization, the proposed approach is shown in several path planning experiments to provide significant improvements in terms of computation time, numerical reliability, and objective function value.}

\section{Introduction} \label{sec:intro}
The problem of computing locally optimal paths for autonomous vehicles such as self-driving cars, unmanned aerial vehicles and autonomous underwater vehicles has recently been extensively studied~\cite{lavalle2006planning, paden2016survey}. However, the task of finding (locally) optimal motions for nonholonomic vehicles in narrow environments is still considered difficult~\cite{zhang2018autonomous}. In this paper, optimal path planning is defined as the problem of finding a feasible and collision-free path from the vehicle initial state to a desired goal state, while a specified performance measure is minimized.

There exist several methods to generate optimized paths for autonomous vehicles. One common approach is to use B-splines or Bezier curves for differentially flat systems, either to smoothen a sequence of waypoints~\cite{piazzi2007mmb,yang2010analytical} or as steering functions within a sampling-based motion planner~\cite{oliveira2018combining}. The use of these methods are computationally efficient since the system dynamics can be described analytically. However, these methods are not applicable to non-flat systems, such as, e.g., many truck and trailer systems~\cite{rouchon1993flatness}. Furthermore, it is difficult to optimize the maneuvers with respect to a general performance measure. 

\begin{figure}
	\centering
	\setlength\figureheight{0.27\textwidth}
	\setlength\figurewidth{0.35\textwidth}
	\input{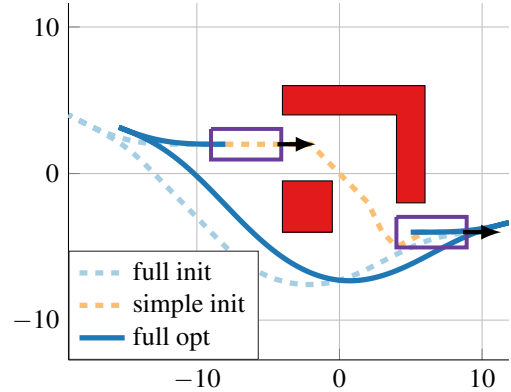}
	\caption{\small \label{fig:infeasible} Feasibility issues. The solution from a path planner based on a simplified geometric model (dashed yellow) provides an initialization from where it is impossible for an \textsc{ocp} solver to find a feasible solution. On the other hand, an initialization based on the full model (dashed blue) enables reliable convergence.}
\end{figure} 

Another popular method is to formulate the problem as an optimal control problem (\textsc{ocp}). One approach is to use mixed integer formulations, which can be computationally demanding to solve~\cite{grossmann2002review}. Another approach is to use numerical optimal control (\textsc{noc}) and transform the problem to a nonlinear program (\textsc{nlp}). Due to non-convex constraints introduced by obstacles and nonlinear system dynamics, a proper initialization strategy is crucial to converge to a good local optimum~\cite{nocedal2006numerical}. A straightforward initialization strategy is to use linear interpolation~\cite{rao2009survey}, this can however often lead to convergence issues in cluttered environments~\cite{bergman2018combining,zhang2018autonomous}. A more sophisticated initialization strategy is to use the solution from a sampling-based path planner. In previous work in the literature, the vehicle models used by the path planner for initialization are simplified; either they disregard the system dynamics completely (geometric planning)~\cite{lavalle2006planning, campos2017hybrid}, or partially (respecting only kinematic constraints)~\cite{andreasson2015fast, zhang2018autonomous}. Using a geometric path planning initialization is shown in~\cite{zhang2018optimization} to cause problems for vehicles with nonholonomic constraints (illustrated in Fig.~\ref{fig:infeasible}). Initializations based on a kinematic model will in general be infeasible in the actual \textsc{ocp} to be solved, and potentially not homotopic to a dynamically feasible solution~\cite{stoneman2014embedding}. Furthermore, the objective function in the sampling-based planner can only consider states that are represented in the chosen simplified model description, which might cause it to find a path far from a local minimum.

A popular deterministic sampling-based path planner is the so called lattice-based path planner, which uses a finite set of precomputed motion segments online to find a resolution optimal solution to the path planning problem~\cite{pivtoraiko2009differentially}. However, since the lattice planner uses a discretized search space, the computed solution can be noticeably suboptimal and it is therefore desirable to improve this solution~\cite{oliveira2018combining,andreasson2015fast}. 

The main contribution of this paper is to combine a lattice-based path planner and \textsc{noc} in a novel way to generate locally optimal solutions to advanced path planning problems. It is done by first introducing a bilevel optimization formulation to analyze and understand the relation between the original optimal path planning problem and the lattice-basted approximation. This way of formulating the problem provides a new tool to analyze and understand the suboptimalities of the lattice-based method. It is shown how the solution from a lattice planner can be improved by initializing a \textsc{noc} solver that efficiently solves the problem to local optimality. The idea is that the lattice-based path planner solves the combinatorial aspects (e.g which side to pass an obstacle) while \textsc{noc} is used to improve the continuous aspects of the solution keeping the combinatorial part fixed. Furthermore, this work goes beyond previous path planning initialization strategies since the initialization is not only feasible, but also optimized using the same objective function as in the improvement step. It is shown in several numerical examples that the proposed approach results in significantly reduced computation time, reliable convergence of the \textsc{noc} solver and generally improved solutions compared to previously used initialization strategies or lattice-based planners alone.

\section{Problem formulation}
In this section, the continuous optimal path planning problem is formulated as an \textsc{ocp}. Furthermore,  we pose a commonly used approximation of the original problem in terms of a lattice-based path planner. 
\subsection{The optimal path planning problem}  
In this work, the optimal path planning problem is defined as follows 
\begin{equation}
\minimize{ J = \int_{0}^{S_g} L(x(s), u(s), q(s)) \mathrm{d}s    }{ u(\cdot), S_g, q(\cdot)  }{&x(0) = x_{\mathrm{init}}, \quad x(S_g)= x_{\mathrm{goal}}, \\
	& x'(s) = f_{q(s)}(x(s), u(s)), \\
	& x(s) \in \mathcalOld{X}_{\mathrm{free}} \cap \mathcalOld{X}_{\mathrm{valid}}, \\
	& q(s) \in \mathcalOld{Q}, \quad u(s) \in \mathcalOld{U}.	  } \label{eq:hybrid_ocp}
\end{equation}
Here, $s > 0$ is defined as the distance traveled by the system, $x \in \mathbf{R}^n$ is the state vector and $u \in \mathbf{R}^m$ is the continuous control input for the system. The derivative of the state with respect to the distance traveled is defined as $\frac{\mathrm{d}x}{\mathrm{d}s} = x'(s)$. $\mathcalOld{X}_{\mathrm{free}}$ represents the obstacle-free region and $\mathcalOld{X}_{\mathrm{valid}}$ and $\mathcalOld{U}$ the feasible sets of states and control signals, respectively. There is also a discrete input signal $q(s) \in \mathcalOld{Q} = \{1, 2, \ldots, N \}$ which enables the selection between $N$ modes of the system. 
It is here assumed that there are finitely many switches in finite distance. 
The system mode determines the vector field $f_{q} \in \mathcalOld{F}$ that describes the current equation of motion~\cite{hedlund1999optimal}. The system mode can for example represent the direction of motion (which is the main use in this paper). However, the results presented in Section~\ref{sec:bilevel}-\ref{sec:improve} also hold for a set $\mathcalOld{F}$ representing a more general switched dynamical system. One such example is morphing aerial vehicles~\cite{falanga2019foldable}. The objective function $J$ to be minimized in \eqref{eq:hybrid_ocp} is specified by the cost function $L(x,u,q)$. This function can depend on the continuous variables as well as the system mode, where the latter enables the possibility of associating each system mode with a unique cost function.

Due to that the problem of solving \eqref{eq:hybrid_ocp} contains the combinatorial aspects of selecting the system mode and the route to avoid obstacles, as well as continuous nonlinear system dynamics, finding a feasible and locally optimal solution is a difficult problem. Hence, approximate methods aiming at feasible, suboptimal solutions are commonly used~\cite{lavalle2006planning}, where the lattice-based path planner provides one alternative. 

\subsection{Lattice-based path planner}
The main idea with a lattice-based path planner is to restrict the controls to a discrete subset of the available actions (motion primitives) and as a result transform the optimization problem \eqref{eq:hybrid_ocp} into a discrete graph search problem. In this paper, the so-called state lattice methodology will be used. The methodology is mainly suitable for position invariant systems, since then motion primitives need only be computed from states with a
position in the origin, and can then be translated to the desired position online~\cite{pivtoraiko2011kinodynamic}. 

The construction of a state lattice can be divided into three steps. First, a desired state space discretization $\mathcalOld{X}_d$ is defined that represents the reachable states in the graph. After the discretization has been selected, the connectivity in the graph is chosen by selecting which neighboring states to connect. Finally, the set of motion primitives $\mathcalOld{P}$ is constructed by generating $K$ motion segments (for example using an \textsc{noc} solver as in~\cite{pivtoraiko2009differentially, ljungqvist2017lattice}) needed to connect the neighboring states, without considering obstacles. A motion primitive \mbox{$m \in \mathcalOld{P}$} is defined as
\begin{equation} \label{eq:motionPrimitive}
m =  \big( x(s), u(s), q \big)  \in  \mathcalOld{X}^{\mathrm{valid}} \times \mathcalOld{U}, \;   s \in [ 0, S],	\end{equation}
and represents a feasible path of length $S$ in a fixed system mode $q \in \mathcalOld{Q}$ that moves the system from an initial state $x(0)\in \mathcalOld{X}_d$ to a final state $x(S)\in \mathcalOld{X}_d$ by applying the control inputs $u(\cdot)$.

After the set of motion primitives has been generated, the original path planning problem \eqref{eq:hybrid_ocp} can be approximated by the following discrete \textsc{ocp}:
\begin{subequations}
	\label{eq:planning}
	\begin{align}
	\minimizer_{\{ m_k , q_k \}_{k=1}^{M},\hspace{0.5ex}M}\hspace{1ex}
	& J_D = \sum_{k=1}^{M} L_u(m_{k}) \label{eq:objLattice}  \\
	\subjecttoo \hspace{2ex}
	& x_1 = x_{\mathrm{init}}, \quad x_{M+1} = x_{\mathrm{goal}},  \\ \label{eq:primStateTrans}
	& x_{k+1} = f_{m}(x_k, m_{k}), \\ \label{eq:primInSet}
	& m_{k}  \in \mathcalOld{P}(x_k, q_k), \\ 
	& q_k \in \mathcalOld{Q}, \\  \label{eq:trajInFreeSpace}
	& c(x_k,m_{k}) \in \mathcalOld{X}_{\mathrm{free}}.
	\end{align} 
\end{subequations}  
The decision variables are the number of phases $M$, the system mode sequence $\{ q_k\}_{k=1}^{M}$ and the applied motion primitive sequence $\{ m_k \}_{k=1}^{M}$. The state transition equation in \eqref{eq:primStateTrans} describes the successor state $x_{k+1}$ after $m_{k}$ is applied from $x_k$ and \eqref{eq:primInSet} ensures that $m_k$ is selected from the set of applicable motion primitives at $(x_k, q_k)$. The constraint in~\eqref{eq:trajInFreeSpace}, ensures that the traveled path of the vehicle when $m_k$ is applied from $x_k$ does not collide with any obstacles. 
Finally, the objective function in \eqref{eq:objLattice} is given by 
\begin{align} \label{eq:obj_f}
L_u(m)  = \int_{0}^{S} L(x(s), u(s), q) \mathrm{d}s.
\end{align}

The use of the approximation in~\eqref{eq:planning} enables the use of efficient graph search methods online (such as A$^*$, ARA$^*$, etc.), making it real-time applicable~\cite{pivtoraiko2009differentially, ljungqvist2017lattice}. The solution is also guaranteed to be dynamically feasible. On the downside, the solutions from a lattice planner often suffer from discretization artifacts~\cite{oliveira2018combining, andreasson2015fast}, making it desirable to smoothen the state lattice solution. Another limitation with graph search is that it is only possible to plan from and to states within the selected state space discretization~\cite{lavalle2006planning}.

\section{Bilevel optimization}\label{sec:bilevel}
In this section, the two problems in~\eqref{eq:hybrid_ocp} and
\eqref{eq:planning} will be related by rewriting the original problem formulation \eqref{eq:hybrid_ocp} into a bilevel optimization problem (\textsc{bop})~\cite{colson2007overview}. It will be shown that this new formulation of the problem allows for an insightful interpretation of the standard lattice solution methodology. In particular it will be used to connect the methodology to parametric optimization, highlight suboptimality properties, and discuss the choice of objective function used at different parts of the lattice-based framework.

A \textsc{bop} is an optimization problem where a subset of the variables are constrained to be an optimal solution to another optimization problem called the lower-level problem. Analogously, the problem on the first level is called the upper-level problem. A general \textsc{bop} can be written as~\cite{colson2007overview}
\begin{equation}
\minimize{F(x,y)}{x, y}{&(x,y) \in \Upsilon \\ 
	&y \in \underset{z }{\argmin} \; \{f(x,z) : (x,z) \in \Omega \}}   \label{eq:genBilelvel}
\end{equation}
where $F(x,y)$ and $f(x,z)$ represent the upper and lower-level objective functions, respectively, and
\begin{align*}
& \Upsilon =  \{ (x,y) \; | \; G_i(x,y) \leq 0, \; i \in  \{1,\ldots,C\} \; \}  \\
& \Omega =  \{ (x,z) \; | \; g_i(x,z) \leq 0, \; i \in  \{1,\ldots, D\} \; \}
\end{align*} 
represent the upper and lower-level feasible sets, which are represented by $C$ and $D$ inequality constraints, respectively. Typically, a subset of the optimization variables in the upper-level problem are considered as parameters to the lower-level problem. Seen from the upper-level problem, the optimality requirement of the lower-level problem is in general a non-convex constraint. Comparably simple examples of bilevel problems, e.g., where the problems on both levels are quadratic programming problems, can be solved by representing the solution to the lower level problem by, e.g., encoding the \textsc{kkt} conditions using mixed integer optimization~\cite{colson2007overview} or explicitly representing the lower-level solution parametrically using a solution obtained from parametric programming~\cite{faisca2007parametric}. It will be shown in this work that the lattice planner can be interpreted as a way of solving a bilevel formulation of~\eqref{eq:hybrid_ocp} using the latter alternative, i.e., representing the lower-level solution explicitly as a (sampled) parametric solution.

\subsection{Bilevel optimization problem reformulation}
It will now be shown how the path planning problem in \eqref{eq:hybrid_ocp} can be reformulated as a \textsc{bop}. Let $L_u(m)$ from \eqref{eq:obj_f} represent the upper-level objective function and introduce lower-level cost function $L_l(x, u, q)$. Assume that 
\begin{align} \label{eq:objectiveRelation}
L_u(m) = \int_{0}^{S}L_l\big(x(s), u(s), q\big)\mathrm{d}s.
\end{align} 
After dividing the path planning problem in \eqref{eq:hybrid_ocp} in $M$ path segments where along each one the system mode is kept constant, it is possible to cast it as an equivalent bilevel (dynamic) optimization problem in the form
\begin{equation}
\minimize{J_u = \sum_{k=1}^{M}L_u\big(m_k\big) }{\{x^0_k, x^g_k, q_k, m_k\}_{k=1}^{M},\hspace{0.5ex}M}{&x^0_1 = x_{\mathrm{init}}, \quad x^g_{M}= x_{\mathrm{goal}}, \\ 
	&x_{k}^0 = x_{k-1}^{g}, \quad q_k \in \mathcalOld{Q}, \\	 
	& m_k \in \underset{(x,u,\bar{q},S)}{\argmin} \text{ } \eqref{eq:llup}}   \label{eq:ulup}
\end{equation}
where the intial state $x^0_k$, final state $x^g_k$ and system mode
$q_k$ for phase $k$ are the upper-level optimization variables
considered as parameters to the lower-level optimization problem. The
constraints given by $x^{0}_{k} =x^{g}_{k-1}$ ensure that the path is
continuous between adjacent path segments. Furthermore, the
corresponding lower-level optimization problem in \eqref{eq:ulup} can
formally be specified as the following multi-parametric \textsc{ocp} (mp-\textsc{ocp})
\begin{equation}  \label{eq:llup}
\begin{array}{rl}
{J_l^*(x^0, x^g, q) = \underset{u(\cdot), x(\cdot), \bar{q}, S}{\textrm{minimize}}} &  \begin{aligned}[t] 
\int_{0}^{S}L_l\big(x(s), u(s), \bar{q}\big) \mathrm{d}s
\end{aligned}  \\[15pt]
\textrm{subject to} &
\begin{aligned}[t] 
& x(0) = x^0, \quad x(S) = x^g,  \\
& x'(s) =  f_{\bar{q}}(x(s), u(s)), \\
& x(s) \in \mathcalOld{X}_{\mathrm{free}} \cap \mathcalOld{X}_{\mathrm{valid}}, \\
& u(s) \in \mathcalOld{U}, \quad \bar{q} = q. 
\end{aligned}
\end{array}
\end{equation}
where $x^0$, $x^g$, and $q$ are considered as parameters from the upper-level problem. Note the similarities between this problem and \eqref{eq:hybrid_ocp}. Here, the main difference is that the system mode is fixed and the path length $S$ typically shorter than $S_g$. 

Above, it was assumed that the objective functions are related as in \eqref{eq:objectiveRelation},  which was necessary in order for the equivalence between \eqref{eq:hybrid_ocp} and \eqref{eq:ulup} to hold. An alternative is to select the objective functions in the two levels more freely in a way that does not satisfy \eqref{eq:objectiveRelation}, with the price of breaking the equivalence between \eqref{eq:hybrid_ocp} and \eqref{eq:ulup}. If such a choice is still made, the solution to \eqref{eq:ulup} with \eqref{eq:llup} will in general no longer be an optimal solution to \eqref{eq:hybrid_ocp}. However, the use of different objective functions allows in practice for a division of the specification of the problem such as finding a minimum time solution by combining, e.g., low-lateral-acceleration solutions from the lower-level problem~\cite{ljungqvist2017lattice}. A bilevel interpretation of this is that the lower-level problem restricts the family of solutions the upper-level problem can use to compose an optimal solution.

\subsection{Analysis of solution properties using bilevel arguments}\label{sec:latticeConnect}
From a practical point of view, the \textsc{bop} consisting of \eqref{eq:ulup} and \eqref{eq:llup} is in principle harder to solve than the standard formulation of the optimal control problem in \eqref{eq:hybrid_ocp}. However, the formulation as a bilevel problem introduces possibilities to approximate the solution by sampling the solution to the lower-level mp-\textsc{ocp} as a function of its parameters. The result of this sampling is that the solution to \eqref{eq:llup} is only computed for $K$ predefined parameter combinations $(x^0_i, x^{g}_i, q_i) \in \mathcalOld{A}, \; i \in \{1, \ldots, K\}$, where $\mathcalOld{A}$ is the user-defined set of combinations. These motion segments obtained by solving the mp-\textsc{ocp} for $K$ parameters together constitute the motion primitive set $\mathcalOld{P}$ used in~\eqref{eq:planning}. An interpretation of this procedure is hence that $\mathcalOld{P}$ used in a lattice planner is a coarsely sampled parametric solution to the mp-\textsc{ocp} in~\eqref{eq:llup} which can be used to represent the optimal solution of the lower-level problem when the upper-level problem is solved. The sampling introduce the well-known suboptimality of only being able to find solutions within the selected discretization~\cite{lavalle2006planning}. However, this approximation makes it possible to solve the bilevel problem in real-time in the form of a lattice planner.

To be able to compute the motion primitives offline, the obstacle avoidance constraints in~\eqref{eq:llup} are disregarded in the lower-level problem and instead handled during online planning in the upper-level problem. After this rearrangement, the \textsc{bop} in~\eqref{eq:ulup} is equivalent to the lattice formulation~\eqref{eq:planning}. In the following theorem, it is shown that this rearrangement of constraints makes it impossible to obtain an optimal solution within the selected discretization since the lower-level problems are not required to satisfy the obstacle avoidance constraints.

\begin{mythe}
	Let $\mathcal{P}_1$ denote the \textsc{bop}
	\begin{equation} \label{eq:thmBilevelOrig}
	\minimize{F(x,y)}{x , y }{ 
		&y \in \underset{z}{\argmin} \; \{F(x,z) : (x,z) \in \Omega\} }
	\end{equation}
	with optimal objective function value $F(x^*_1, y^*_1)$. Furthermore, let $\mathcal{P}_2$ denote the \textsc{bop}
	\begin{equation} \label{eq:thmBilevelSimplified}
	\minimize{F(x,y)}{x , y }{& (x,y) \in \Omega \\
		&y \in \underset{z}{\argmin} \; \{F(x,z) \} }
	\end{equation} 
	with optimal objective function value $F(x^*_2, y^*_2)$. It then holds that $F(x^*_1, y^*_1)\leq F(x^*_2, y^*_2)$.
\end{mythe}

\begin{proof}
	The feasible set of $\mathcal{P}_1$ is $Z_1 = \{(x,y) \; | \; y \in \argmin _z \; \{F(x,z) : (x,z) \in \Omega \} \}$, and the feasible set of $\mathcal{P}_2$ is $Z_2 = \{(x,y) \; | \; \big( y \in \argmin _z \; \{F(x,z) \} \big) \cap \Omega \; \}$. Hence, any point in $Z_2$ is also in $Z_1$, i.e., $Z_2 \subseteq Z_1 \implies$ $F(x^*_1, y^*_1)\leq F(x^*_2, y^*_2)$.
\end{proof}

\begin{mydef}
	The active set $\mathcalOld{A}(x,y)$ at a feasible pair $(x,y)$ of \eqref{eq:thmBilevelOrig} consists of the inequality constraints that hold with equality~\cite{nocedal2006numerical}, i.e.,
	\begin{equation*}
	\mathcalOld{A}(x,y) = \{i \in \{1 \ldots D \} \; | \; g_i(x,y) = 0 \}.
	\end{equation*}
\end{mydef}

\begin{mydef}
	Let $(x^*, y^*)$ be an optimal solution to an optimization problem with \textsc{kkt} conditions satisfied with Lagrange multipliers $\lambda^*$ associated to the inequality constraints in $\Omega$. A constraint $g_i(x,y)$ in $\Omega$ is then said to be \textit{strongly active} if $g_i(x^*, y^*) = 0 $ and $\lambda^*_i > 0$~\cite{nocedal2006numerical}. 
\end{mydef}

\begin{mycor}
	\label{cor:ucnstrsubopt}
	Assume that the optimal solution $(x_1^*, y_1^*)$ to $\mathcal{P}_1$ in~\eqref{eq:thmBilevelOrig} is unique. Then, if there exists an $i$ such that $g_i(x_1^*, y_1^*)$ is strongly active in the lower-level problem, it holds that $F(x^*_1, y^*_1) < F(x^*_2, y^*_2)$ where $(x_2^*, y_2^*)$ is the optimal solution to $\mathcal{P}_2$ in~\eqref{eq:thmBilevelSimplified}.
\end{mycor}

\begin{proof}
	Since there exists at least one constraint which is strongly active at the lower level, it follows that $(x^*_1, y^*_1) \notin Z_2$, since $(x_1^*, \argmin_z \; \{F(x_1^*,z) \}) \notin \Omega$. Hence, $(x_1^*, y_1^*) \in Z_1\setminus Z_2$. Since $(x_1^*, y_1^*)$ is the unique optimal solution to $\mathcal{P}_1$ over \mbox{$Z_1 \supseteq Z_2$}, it follows that $\nexists \; (x_2^*,y_2^*) \in Z_2:F(x^*_2, y^*_2) \leq F(x^*_1, y^*_1)$. Hence, $F(x^*_1, y^*_1) < F(x^*_2, y^*_2)$.
\end{proof}

An interpretation of Corollary~\ref{cor:ucnstrsubopt} is that if the optimal solution to~\eqref{eq:ulup} is ``strongly'' in contact with the environment, then it is not in general possible to obtain an optimal solution using solutions to the lower-level problem (i.e., motion primitives) computed \emph{without considering obstacles}. Note that these effects are beyond the fact that lower-level problems are sampled on a grid. The lower-level family of solutions is no longer optimal, instead the solutions need to adapt to the surrounding environment to become optimal, which is not a part of the standard lattice planning framework.

It will now be shown that the consequences of the suboptimality aspects discussed in this section of the approximate solution to \eqref{eq:hybrid_ocp} obtained by solving~\eqref{eq:planning} using a lattice planner, can be efficiently reduced using \textsc{noc} employing the solution from the lattice planner as a good warm-start.

\section{Improvement using numerical optimal control}\label{sec:improve}
In this section, we propose to use \textsc{noc} to improve the approximate solution computed by the lattice planner. By letting the system mode sequence $\sigma =\{q_k\}_{k=1}^M$ be fixed to the solution from the lattice planner, the following \textsc{ocp} is obtained:

\begin{equation}
\minimize{\sum_{k=1}^{M} \int_{0}^{S_k}L\big(x_k(s), u_k(s), \sigma[k] \big)\mathrm{d}s  }{ \{u_k(\cdot), \; S_k \}_{k=1}^M}{&x^1(0) = x_{\mathrm{init}}, \quad x_k(S_M)= x_{\mathrm{goal}}, \\
	& x'_{k}(s) = f_{\sigma[k]}(x_k(s), u_k(s)), \\
	& x_{k+1}(0) = x_k(S_{k}), \\
	& x_k(s) \in \mathcalOld{X}_{\mathrm{free}} \cap \mathcalOld{X}_{\mathrm{valid}}, \quad u_k(s) \in \mathcalOld{U},
} \label{eq:online_smoothing}
\end{equation}
where the optimization variables are the control signals $u_k(\cdot)$ and lengths $S_k$ of the $M$ phases. The difference compared to the optimal path planning problem \eqref{eq:hybrid_ocp} is that the combinatorial aspect of selecting the system mode sequence is already specified. However, since the length of the phases are optimization variables, it is possible that redundant phases introduced by the lattice planner are removed by selecting their lengths to zero. Furthermore, the second combinatorial aspect of selecting how to pass obstacles is implicitly encoded in the warm-start solution from the lattice planner.

The problem in~\eqref{eq:online_smoothing} is in the form of a standard multiphase \textsc{ocp}, where the subsequent phases are connected using equality constraints. This problem can be solved using \textsc{noc}, for example by applying a direct method to reformulate the problem as an \textsc{nlp} problem~\cite{rao2009survey}. Today, there exist high-performing open-source \textsc{nlp} software such as \textsc{ipopt}~\cite{wachter2006implementation}, \textsc{wohrp}~\cite{bueskens2013worhp}, etc., that can be used to solve these types of problems. 
Common for such \textsc{nlp} solvers is that they aim at minimizing both the constraint violation and the objective function value~\cite{nocedal2006numerical}. Hence, a good initialization strategy should consider both the objective function and feasibility. In this work, the resolution optimal solution $\{m_k\}_{k=1}^M$ from the lattice planner is used to initialize the \textsc{nlp} solver. It represents a path that is not only dynamically feasible, but also where each phase in the path (i.e each motion primitive) has been computed by optimizing the \emph{same} cost function $L(x,u,q)$ as in~\eqref{eq:online_smoothing}. Hence, the \textsc{nlp} solver is provided with a well-informed warm-start,  which in general will decrease the time for the \textsc{nlp} solver to converge to a locally optimal~solution~\cite{nocedal2006numerical}. Furthermore, when the same objective function is used both in \eqref{eq:planning} and \eqref{eq:online_smoothing}, the \textsc{nlp} solver will in general be initialized close to a good local minimum. Finally, a benefit of using a feasible initialization is that it is \emph{always} guaranteed that a dynamically feasible solution exists that is homotopic with the provided initialization (at the very least the initialization itself), making it reliable to use online.  Due to all these properties, the step of solving~\eqref{eq:online_smoothing} is referred to as an improvement step in this work, which is somewhat in contrast to previous work where this step is commonly denoted ``smoothing''. Its primary aim is to improve the solution obtained from the lattice planner in terms of improving the objective function value. 
\vspace{2pt}
\begin{myrem}
	Note that the improvement step also can be used to enable path planning from and to arbitrary initial and goal states that are not within the specified state space discretization $\mathcalOld{X}_d$. Here, the lattice planner can be used to find a path from and to the closest states in $\mathcalOld{X}_d$, and the improvement step can then adapt the path such that it starts at the initial state and reaches the goal state exactly. However, in this case the warm-start cannot be guaranteed to be feasible.
\end{myrem}
\vspace{2pt}
After the improvement step is applied, a solution with lower objective function value compared to the solution from the lattice planner will in general be obtained since, relating back to Section~\ref{sec:latticeConnect}, the discretization constraints in the bilevel formulation are removed, and the paths are constructed while explicitly considering obstacles. 

\subsection{Proposed path planning approach}

\begin{algorithm}[t]
	\caption{Proposed path planning approach}
	\label{alg:ppas}
	\begin{algorithmic}[1]
		\State \textbf{Offline}:
		\State \textsc{input}: $x(s)$, $\mathcalOld{X}_{\mathrm{valid}}$, $u(s)$,  $\mathcalOld{U}$, $q \in \mathcalOld{Q}$, $f_q (x(s),u(s)) \in \mathcalOld{F}$ and $L_u(m)$, $ L_l(x, u, q)$ and $L(x, u, q)$.
		\State Choose $\mathcalOld{X}_d$ and select $(x_i^0, x_i^g, q_i), \; i = \{1 \ldots K\}$ 
		\State $\mathcalOld{P} \leftarrow$ solve $K$ \textsc{ocp}s \eqref{eq:llup} disregarding obstacle constraints. 
		\State \textbf{Online}: 
		\State \textsc{input}: $x_{\mathrm{init}}, \; x_{\mathrm{goal}}, \; \mathcalOld{X}_{\mathrm{free}}$. 
		\State \textsc{lattice planner} : $\{ m_k, q_k \}_{k=1}^M$ $\leftarrow$ Solve \eqref{eq:planning} from  $x_{\mathrm{init}}$ to $x_{\mathrm{goal}}$ with $\mathcalOld{P}$.
		\State \textsc{improvement} : $\{ x_k(\cdot), u_k(\cdot), S_k \}_{k=1}^M$ $\leftarrow$ Solve \eqref{eq:online_smoothing} with $ \sigma = \{q_k \}_{k=1}^M$, warm-started with $\{ m_k \}_{k=1}^M$.

	\end{algorithmic}
\end{algorithm}
A solution to the path planning problem is found using a preparation step offline and a two-step procedure online according to Algorithm~\ref{alg:ppas}. In the preparation step, the objective functions used in the motion primitive generation and graph search in the lattice planner and the improvement step are specified. Furthermore, the system modes with associated vehicle models (used in the motion primitive generation and improvement step) are defined. Then, the motion primitive set $\mathcalOld{P}$ is computed by solving the \textsc{ocp}s defined by the user without considering obstacles. For a detailed explanation of this step, the reader is referred to~\cite{pivtoraiko2009differentially, bergman2019improvedArxiv}. 

The first step online is called whenever a new path planning problem from $x_{\mathrm{init}}$ to $x_{\mathrm{goal}}$ should be solved. In this step, a lattice planner is used to solve the approximate path planning problem~\eqref{eq:planning} by using the precomputed motion primitive set $\mathcalOld{P}$ and the current description of the available free space $\mathcalOld{X}_{\mathrm{free}}$. The solution is a resolution optimal path, where the system mode is kept constant in each phase. This solution is used as a well-informed warm-start to the final improvement step, where the multiphase \textsc{ocp} in \eqref{eq:online_smoothing} is solved to locally optimality by improving the continuous aspects of the solution from the lattice planner.

\section{Numerical Results} \label{sec:Res}
In this section, the proposed path planning approach is applied to two different
vehicular systems; a car and a truck and trailer system. 
During online planning, the lattice planner is implemented using A$^*$ graph search, where a precomputed free-space heuristic look-up table (\textsc{hlut})~\cite{knepper2006high} is used as heuristic function to guide the search process. The \textsc{hlut} is computed by solving path planning problems in an obstacle-free environment from all initial states $x \in \mathcalOld{X}_d$ with a position at the origin to all final states $x \in \mathcalOld{X}_d$ with a position within a square centered around the origin with side length $\rho$ (in the experiments, $\rho=40$ m for the car and $\rho=80$ m for the truck and trailer system). The motion primitive generation and the improvement step are both implemented in Python using Cas\textsc{ad}i~\cite{andersson2018casadi}, where the warm-start friendly \textsc{sqp} method \textsc{worhp} is used as \textsc{nlp} solver. All simulations are performed on a laptop computer with an Intel Core i7-5600U processor.
\subsection{Vehicle models}
The model of the car is based on a kinematic bicycle model~\cite{lavalle2006planning} with state vector
$x_c(s) = (\bar{x}_c(s),  \alpha(s), \omega(s) )$, where  $\bar{x}_c(s) = (x_1(s), y_1(s), \theta_1(s))$. Here, $(x_1,y_1)$ is the center of the rear axle of the car, $\theta_1$ is the car's orientation, $\alpha$ is the front-wheel steering angle and $\omega$ is the steering angle rate. The vehicle model is
\begin{equation} \label{eq:carModel}
\begin{aligned}
&\bar{x}_c'(s) = q \left(\cos \theta_1(s), \sin \theta_1(s), \frac{\tan \alpha(s)} { L } \right)^T,\\
& \alpha '(s) = \omega(s), \quad \omega'(s) = u_{\omega}(s),
\end{aligned}
\end{equation}
where $u_{\omega}$ is the continuous control signal to the system which represents the steering angle acceleration, $L=2.9$ m the wheel-base of the car and $q \in \{1, -1\}$ is the discrete decision variable representing the direction of motion. The constraints imposed on the states and control signal are given by $ |\alpha(s)| \leq \pi/4$, $|\omega(s)| \leq 0.5$ and $|u_{\omega}(s)| \leq 40$. The cost function used for the car is given by
\begin{equation} \label{eq:ObjectiveCar}
L_c(x_c, u_{\omega},q) = 
1 + \gamma (\alpha^2 + 10\omega^2 + u _{\omega} ^2),   \\ 
\end{equation}
where the variable $\gamma$ represents the trade-off between path length and smoothness of the solution. The truck and trailer system is a general 2-trailer with a car-like truck~\cite{altafini2002hybrid}. This system is composed of three interconnected vehicle segments; a car-like truck, a dolly and a semitrailer. The state vector for this system is given by
$x_t(s) = (\bar{x}_t(s),  \alpha(s), \omega(s) )$ where $\bar{x}_t(s) = (x_3(s), y_3(s), \theta_3(s), \beta_3(s), \beta_2(s))$. Here, $(x_3,y_3)$ is the center of the axle of the semitrailer, $\theta_3$ is the orientation of the semitrailer, $\beta_3$ is joint angle between the semitrailer and the dolly, $\beta_2$ is joint angle between the dolly and the car-like truck. The truck's steering angle $\alpha$ and its derivatives, $\omega$ and $u_\omega$, are subject to the same constraints as in the car-case. The model of this system can compactly be represented as (see~\cite{ljungqvist2017lattice} for details):
\begin{equation} \label{eq:truckModel}
\begin{aligned}
&\bar{x}_t'(s) = q f_t(\bar{x}_t(s), \alpha(s) ), \\
& \alpha '(s) = \omega(s), \quad \omega'(s) = u_{\omega}(s),
\end{aligned}
\end{equation}
where $q \in \{1, -1\}$ also in this case represents the direction of motion. The model parameters used in this section coincide with what is used in~\cite{ljungqvist2017lattice}. 
The cost function used for the truck and trailer system is given by
\begin{equation} \label{eq:ObjectiveTruck}
\small
\begin{aligned}
&L_t(x_t, u_{\omega}, q) = 
& \begin{cases}
1 + \gamma (\alpha^2 + 10\omega^2 + u _{\omega} ^2), &q = 1,  \\ 
1 + \gamma (\beta_3^2 + \beta_2^2 + \alpha^2 + 10\omega^2 + u _\omega ^2), &q = -1, 
\end{cases}
\end{aligned} 
\end{equation}
i.e., quadratic penalties for large joint angles $\beta_3$ and $\beta_2$ are added to the cost function for paths in backward motion to avoid so-called jack-knife configurations. Unless stated otherwise, $\gamma = 1$ is used in both \eqref{eq:ObjectiveCar} and \eqref{eq:ObjectiveTruck}.

\subsection{State lattice construction}
To illustrate the full potential of the proposed approach, three different motion primitive sets for each vehicle are used by the lattice planner, where the sets use either simplified or complete vehicle models. The first motion primitive sets $\mathcalOld{P}_{\mathrm{dyn}}$ use the complete vehicle models. The second sets $\mathcalOld{P}_{\mathrm{kin}}$ disregard $\omega$ and $u_\omega$, making the steering angle $\alpha$ considered as control signal (i.e., purely kinematic models), which is similar to the initialization strategy used in~\cite{zhang2018autonomous}. The third sets $\mathcalOld{P}_{\mathrm{geo}}$ are computed by completely neglecting the system dynamics, further referred to as a geometric model, where instead linear interpolation is used between the initial and final states for each motion primitive. 

Before computing the motion primitive sets, the state space of the vehicles need to be discretized. The positions, $(x_1,y_1)$ for the car and $(x_3,y_3)$ for the semitrailer, are discretized onto a uniform grid with resolution $r=1$ m and the orientations $\theta_1\in\Theta$ and $\theta_3\in\Theta$ are irregularly discretized as proposed in~\cite{pivtoraiko2009differentially}. The discretization of the steering angle $\alpha$ is only applicable for the complete models. For simplicity, it is here constrained to zero and its rate $\omega$ is also constrained to zero to ensure that $\alpha$ is continuously differentiable, even when motion segments are combined online~\cite{ljungqvist2017lattice}. For the truck and trailer system, the joint angle $\beta_3$ and $\beta_2$ are also constrained to zero at each discretized state in the state lattice. Note however that on the path between two discretized states, the systems can take any feasible vehicle configuration. 

The motion primitive sets are automatically computed using the approach described in~\cite{bergman2019improvedArxiv}, where the sets are composed of heading changes and parallel maneuvers according to Table~\ref{tab:primSets}. These maneuvers are optimized using the cost functions defined in~\eqref{eq:ObjectiveCar} and~\eqref{eq:ObjectiveTruck}. For the simplified vehicle models, the neglected states are disregarded in the cost functions. For a more detailed description of the state lattice construction, the reader is referred to~\cite{bergman2019improvedArxiv}.

\begin{table}[t]
	\caption{ \small A description of the different motion primitive sets used. $\lvert \Theta \rvert $ defines the number of heading discretization points, $\Delta_{\theta}^{\mathrm{max}}$ defines which neighboring headings to connect (from 1 to $\Delta_{\theta}^{\mathrm{max}}$) and $n_{\mathrm{par}}$ defines the number of parallel maneuvers (per heading). Finally, $n_{\mathrm{prim}}$ defines the resulting total number of motion primitives.} \label{tab:primSets}	
	\normalsize
	\centering
	\begin{tabular}{ccccc}	
		$\mathcalOld{P}$ &  $\lvert \Theta \rvert $ & $\Delta_{\theta}^{\mathrm{max}}$ & $n_{\mathrm{par}}$ & $n_{\mathrm{prim}}$ \\
		\hline
		$\mathcalOld{P}_{\mathrm{geo}}$ & 16 & 2  & N/A & 240   \\ 
		$\mathcalOld{P}_{\mathrm{kin}}$  & 16 & 4 & 3   & 480  \\ 
		$\mathcalOld{P}_{\mathrm{dyn}}$  & 16 & 4 & 3   & 480 \\
		\hline
	\end{tabular}
\end{table}

\subsection{Experimental results}
For the car model, two different path planning scenarios are considered; a parallel parking problem (Fig.~\ref{fig:par_park}) and one with multiple routes to avoid obstacles (Fig.~\ref{fig:mult_hom}). For the truck and trailer system, a loading site area is used (Fig.~\ref{fig:loading_site}). The obstacles and vehicles are implemented using bounding circles~\cite{lavalle2006planning}; the area of the car is described by three circles, while the truck is described by one circle and the trailer by two circles. This choice of obstacle representation can be used in all steps since the constraints can be described by smooth functions. An alternative object representation that is perfectly compatible with the approach presented in this work is proposed in~\cite{zhang2018optimization}, where vehicles and obstacles can be represented by general convex sets. 

The path planning problems are first solved by the lattice planner, using the three different motion primitive sets described in Table~\ref{tab:primSets}. Thereafter, the obtained solutions are used to initialize the improvement step. For the simplified models, all states that are not represented are initialized to zero.

For the car scenarios in Fig.~\ref{fig:par_park}-\ref{fig:mult_hom}, the results in Table~\ref{tab:par_park}-\ref{tab:mult_hom} show that the lattice planner achieves the lowest computation times if the geometric model is used, compared to the kinematic and the complete model. However, using this simple initialization strategy results in a decreased reliability (only 62 \% and 54.5 \% success rate) and the total average computation time becomes higher than the two other cases due to a more computationally demanding improvement step. The kinematic initialization performs better than the geometric in terms of reliability, but in a cluttered environment (Table~\ref{tab:mult_hom}) the success rate is only 75.3 \%. When the complete model is used in the lattice planner, the computation time for the improvement step is significantly reduced compared to when the simpler initialization strategies are used. In particular, the total computation time including the lattice planner is as much as halved and the success rate is always 100~\%. Furthermore, the mean objective function value $\overline{J}_{\mathrm{opt}}$ decreases significantly compared to the solution from the lattice planner $\overline{J}_P$. For the two simpler initialization strategies, no comparable objective function values from the lattice planner exist since they are infeasible to the actual path planning~problem.

\begin{figure}
	\centering
	\setlength\figureheight{0.18\textwidth}
	\setlength\figurewidth{0.36\textwidth}
	\input{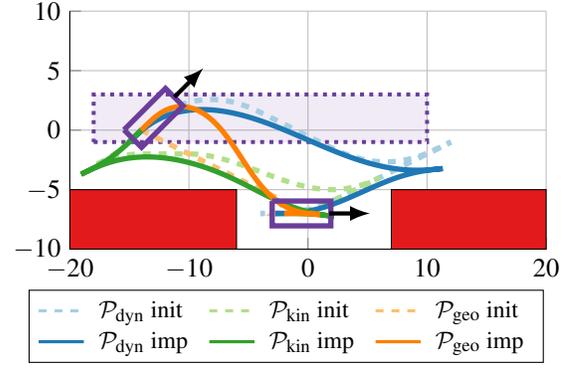}
	\definecolor{mycolor1}{rgb}{0.6510,0.8078,0.8902}%
\definecolor{mycolor2}{rgb}{0.1216,0.4706,0.7059}%
\definecolor{mycolor3}{rgb}{0.6980,0.8745,0.5412}%
\definecolor{mycolor4}{rgb}{0.2000,0.6275,0.1725}%
\definecolor{mycolor5}{rgb}{0.9922,0.7490,0.4353}%
\definecolor{mycolor6}{rgb}{1.0000,0.4980,0}%
\begin{tikzpicture}

\begin{axis}[%
hide axis,
width=0,
height=0,
at={(0,0)},
scale only axis,
xmin=0,
xmax=1,
ymin=0,
ymax=1,
axis background/.style={fill=white},
xmajorgrids,
ymajorgrids,
legend style={legend cell align=left,column sep=3pt, align=center},
legend columns=2,transpose legend
]
\addplot [color=mycolor1, line width=1.2pt, dashed]
table[row sep=crcr]{%
	1	0\\
};
\addlegendentry{ {\small $\mathcalOld{P}_{\mathrm{dyn}}$ init} }

\addplot [color=mycolor2, line width=1.2pt]
table[row sep=crcr]{%
	1	0\\
};
\addlegendentry{ {\small $\mathcalOld{P}_{\mathrm{dyn}}$ imp} }

\addplot [color=mycolor3, line width=1.2pt, dashed]
table[row sep=crcr]{%
	1	0\\
};
\addlegendentry{ {\small $\mathcalOld{P}_{\mathrm{kin}}$ init} }

\addplot [color=mycolor4, line width=1.2pt]
table[row sep=crcr]{%
	1	0\\
};

\addlegendentry{ {\small $\mathcalOld{P}_{\mathrm{kin}}$ imp} }

\addplot [color=mycolor5, line width=1.2pt, dashed]
table[row sep=crcr]{%
	1	0\\
};
\addlegendentry{ {\small $\mathcalOld{P}_{\mathrm{geo}}$ init} }

\addplot [color=mycolor6, line width=1.2pt]
table[row sep=crcr]{%
	1	0\\
};
\addlegendentry{ {\small $\mathcalOld{P}_{\mathrm{geo}}$ imp} }

\end{axis}
\end{tikzpicture}%
	\caption{\small \label{fig:par_park} Parallel parking scenario solved from several initial states with $\theta_1^i = \{0, \pi/4\}$ (indicated by area within the dotted lines). Solutions from one problem are illustrated for the three initialization strategies (using the motion primitive sets described in Table~\ref{tab:primSets}) with corresponding solutions from the improvement step. }
\end{figure} 

\begin{figure}
	\centering
	\setlength\figureheight{0.25\textwidth}
	\setlength\figurewidth{0.32\textwidth}
	\input{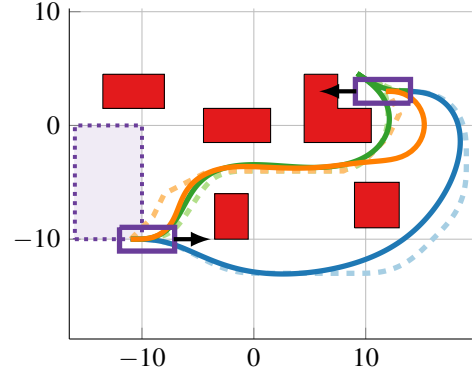}
	\caption{\small \label{fig:mult_hom} A problem with solved from several initial states with $\theta^i_1 = 0$ with multiple routes to get to the goal state $(x^g_1,y^g_1,\theta^g_1, \alpha^g) = (13,3,\pi,0)$. See Fig.~\ref{fig:par_park} for a further description of the content. }
\end{figure} 
\begin{table}[t]
	\caption{ \small Results from parallel parking scenario (Fig.~\ref{fig:par_park}, 150 problems). $\mathcalOld{P}$ is the motion primitive set used in the lattice planner. $\overline{t}_{P}$ is the average time for the lattice planner to find a solution. $r_{\mathrm{imp}}$ and $\overline{t}_{\mathrm{imp}}$ are the success rate and average time for the improvement step to converge. $\overline{t}_{\mathrm{tot}}$ is the average total time. Finally, $\overline{J}_{\mathrm{P}}$ and $\overline{J}_{\mathrm{imp}}$ is the average objective function value for the solutions from the lattice planner and improvement step, respectively.  } \label{tab:par_park}	
	\normalsize
	\centering
	\begin{tabular}{ccccccc}	
		$\mathcalOld{P}$ & $\overline{t}_{P}$ [s] & $\overline{t}_{\mathrm{imp}}$ [s] & $\overline{t}_{\mathrm{tot}}$ [s] & $r_{\mathrm{imp}}$  &$\overline{J}_P$ & $\overline{J}_{\mathrm{imp}}$   \\
		\hline
		$\mathcalOld{P}_{\mathrm{geo}}$ & 0.0011 & 1.12 & 1.12   & 62 \% & N/A & 30.8  \\
		$\mathcalOld{P}_{\mathrm{kin}}$ & 0.025  & 1.03  & 1.06  & 90.7 \% & N/A & 28.7 \\
		$\mathcalOld{P}_{\mathrm{dyn}}$ & 0.014 & 0.88  & 0.894  & 100 \% & 35.7 & 27.5  \\ 
		\hline
	\end{tabular}
\end{table}

\begin{table}[t]
	\caption{\small Results from multiple routes scenario (Fig.~\ref{fig:mult_hom}, 77 problems). See Table~\ref{tab:par_park} for a description of the variables.  } \label{tab:mult_hom}	
	\normalsize
	\centering
	\begin{tabular}{ccccccc}	
		$\mathcalOld{P}$ & $\overline{t}_{P}$ [s] & $\overline{t}_{\mathrm{imp}}$ [s] & $\overline{t}_{\mathrm{tot}}$ [s] & $r_{\mathrm{imp}}$ &$\overline{J}_{P}$ & $\overline{J}_{\mathrm{imp}}$   \\
		\hline
		$\mathcalOld{P}_{\mathrm{geo}}$ & 0.03  & 4.82 & 4.85 & 54.5 \% & N/A & 60.5  \\
		$\mathcalOld{P}_{\mathrm{kin}}$ & 0.36  & 3.81 & 4.17  & 75.3 \% & N/A & 50.0 \\
		$\mathcalOld{P}_{\mathrm{dyn}}$ & 0.31 & 2.14 & 2.45 & 100 \% & 59.9 & 50.2  \\ 
		\hline
	\end{tabular}
\end{table} 
The results for the truck and trailer system (Fig.~\ref{fig:loading_site}) are summarized in Table~\ref{tab:load_1}. In this experiment, using a dynamically feasible initialization (as proposed in this work) has an even larger impact on the time spent in the improvement step; the average time using the geometric and kinematic models has dropped from 18.4 s and 15 s, respectively, down to 9.5 s for the dynamic model. 
The reason why such a large computational performance gain is obtained in this more advanced scenario is mainly due to the complicated system dynamics, which also affect the reliability using a geometric initialization strategy where the success rate is less than \mbox{$50$ \%}. Finally, the reliability for the kinematic initialization is higher compared to the car scenarios. This is mainly due to a less cluttered environment, which enables a higher success rate for the kinematic initialization strategy.

In Table~\ref{tab:load_objective} the impact of using the complete vehicle model and the same or different objective functions at the three steps; the motion primitive generation, the graph search in the lattice planner and the improvement step is analyzed. The results on row 1 ($\gamma_u = \gamma_l = \gamma_i = 10$) represent the baseline where the same objective function is used in all steps. When shortest path is used as objective function in the graph search ($\gamma_u = 0$, row 2 in Table~\ref{tab:load_objective}), the average cost for a path is increased by roughly 10 \%, due to that the improvement step converges to a worse local minimum. However, the total computation time decreases, as a result of a faster graph search. This is mainly due to that Euclidean distance is used as heuristic function outside the range of the \textsc{hlut}, which is a better estimator of cost to go when solely shortest path is used as objective function. The computation time for the improvement step is similar to using $\gamma_u = 10$, which is reasonable since each phase (i.e. motion primitive) in the warm-start is optimized using the same objective function as in the improvement step. When also the motion primitives are generated using shortest path as objective function ($\gamma_u = \gamma_l = 0$, row 3 in Table~\ref{tab:load_objective}), not only the average solution cost increases, but also the convergence time for the improvement step. The reason is that each phase in the initialization is far from a local minimum in terms of the objective function used in the improvement step. This clearly illustrates the importance of using the same objective function in the motion primitive generation and improvement step for fast convergence in the latter step.

\begin{table}[h]
	\caption{\small Results from loading site scenario (Fig.~\ref{fig:loading_site}, 270 problems). See Table~\ref{tab:par_park} for a description of the variables.   } \label{tab:load_1}	
	\normalsize
	\centering
	\begin{tabular}{ccccccc}	
		$\mathcalOld{P}$ & $\overline{t}_{P}$ [s]  & $\overline{t}_{s}$ [s] & $\overline{t}_{\mathrm{tot}}$ [s] & $r_{\mathrm{imp}}$ &$\overline{J}_P$ & $\overline{J}_{\mathrm{imp}}$   \\
		\hline
		$\mathcalOld{P}^{}_{\mathrm{geo}}$  & 0.69 & 18.4 & 19.1 & 43.7 \% & N/A & 236  \\
		$\mathcalOld{P}^{}_{\mathrm{kin}}$  & 6.5 & 15.0 & 21.5 & 98.9 \% & N/A & 164 \\
		$\mathcalOld{P}^{}_{\mathrm{dyn}}$ & 5.35 & 9.45 & 14.8 & 100 \% & 184 & 164  \\ 
		\hline
	\end{tabular}
\end{table}

\begin{figure}
	\centering
	\setlength\figureheight{0.256\textwidth}
	\setlength\figurewidth{0.32\textwidth}
	\input{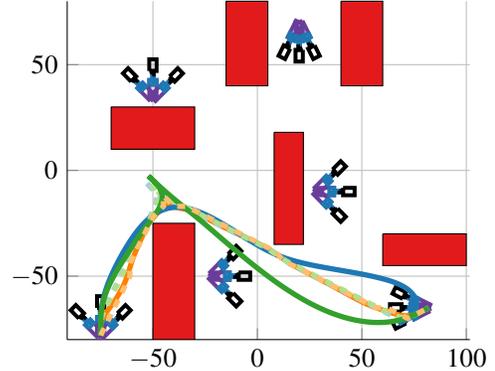}
	\caption{\small \label{fig:loading_site} Loading site scenario using the truck and trailer system, solved from and to several initial and goal states. See Fig.~\ref{fig:par_park} for a further description of the content.}
	\vspace{-5pt} 
\end{figure}

\begin{table}[h!]
	\caption{\small Results from loading site scenario (Fig.~\ref{fig:loading_site}, 270 problems). $\gamma_u$, $\gamma_l$ and $\gamma_{i}$ are the values of $\gamma$ in \eqref{eq:ObjectiveTruck} used in the graph search, the motion primitive generation and improvement step, respectively. See Table~\ref{tab:par_park} for a description of the other variables.    } \label{tab:load_objective}	
	\normalsize
	\centering
	\begin{tabular}{ccccccccc}	
		$\mathcalOld{P}$ & $\gamma_u$ & $\gamma_l$ & $\gamma_i$ & $\overline{t}_{P}$ & $\overline{t}_{\mathrm{s}}$ & $\overline{t}_{\mathrm{tot}}$  & $r_{\mathrm{imp}}$ & $\overline{J}_{\mathrm{imp}}$   \\
		\hline
		$\mathcalOld{P}^{}_{\mathrm{dyn}}$ & 10 & 10  & 10 & 3.2 & 6.2 & 9.4 & 100 \% & 190 \\
		$\mathcalOld{P}^{}_{\mathrm{dyn}}$ & 0  & 10  & 10 & 2.2 & 5.9 & 8.1 & 100 \% & 208 \\
		$\mathcalOld{P}^{}_{\mathrm{dyn}}$ & 0 &  0  &  10 & 1.9  & 12.6 & 14.5  & 100 \% & 212 \\
		\hline
	\end{tabular}
\end{table}
\section{Conclusions and Future Work}
This paper presents a unified optimization-based path planning approach to efficiently compute high-quality locally optimal solutions to path planning problems. The approach is motivated by first showing that a lattice-based path planner can be cast and analyzed as a bilevel optimization problem. This information is then used to motivate a novel, tight combination of a lattice-based path planner and numerical optimal control. 
The first step of the proposed approach consists of using a lattice-based path planner to find a resolution optimal solution to the path planning problem using a discretized search space. This solution is then used as a well-informed warm-start in a second improvement step where numerical optimal control is used to compute a locally optimal solution to the path planning problem. To tightly couple the two steps, the lattice-based path planner uses a vehicle model and objective function that are chosen to coincide with those used in the second improvement step. This combination of a path planner and numerical optimal control makes, in a structured way, benefit of the former method's ability to solve combinatorial parts of the problem and the latter method's ability to obtain locally optimal solutions not restricted to a discretized search space. The value of this tight combination is thoroughly investigated with successful results in several practically relevant path planning problems 
where it is shown to outperform previously existing common initialization strategies in terms of computation time, numerical reliability, and objective function value.

Future work includes to further decrease the online planning time by, e.g., applying the improvement step in a receding horizon fashion or to only optimize parts of the solution obtained from the lattice-based path planner. Another extension is to apply the approach to systems with more distinct system modes, such as morphable drones~\cite{falanga2019foldable}.
\section{Acknowledgments}
This work was partially supported by FFI/VINNOVA and the Wallenberg Artificial Intelligence, Autonomous Systems and Software Program (WASP) funded by Knut and Alice Wallenberg Foundation.
\bibliographystyle{IEEEabbrv}
\bibliography{myrefs.bib}		
\end{document}